\RequirePackage{fix-cm}
\documentclass[smallextended]{svjour3}       % onecolumn (second format)
\smartqed  % flush right qed marks, e.g. at end of proof
\usepackage{amssymb,latexsym}
\usepackage{graphicx}
\usepackage{amsmath}
\usepackage[dvipsnames]{xcolor}
\usepackage{hyperref}
\usepackage{amsmath}
\usepackage{amssymb}

\begin{document}

\title{Generalized Laura-Andoyer equations and the enumeration of some symmetrical classes of Dziobek configurations 
}

\titlerunning{Generalized Laura-Andoyer equations and Symmetrical Dziobek Configurations}        % if too long for running head

\author{ Thiago Dias \and Ya-Lun Tsai       
         %etc.
}

%\authorrunning{} % if too long for running head

\institute{   Thiago Dias \at
           Departamento de Matem\'atica, Universidade Federal Rural de Pernambuco, av. Dom Manoel de Medeiros s/n, Dois Irm\~aos - Recife - PE
52171-900, Brasil
             \email{thiago.diasoliveira@ufrpe.br}           
\and   
Ya-Lun Tsai  \at
 Department of Applied Mathematics, National Chung Hsing University, Taichung, Taiwan 
           \email{yltsai@nchu.edu.tw}    
}

\date{Received: date / Accepted: date}

\maketitle

{\it{The authors dedicate this work to Alain Chenciner on the occasion of his 80th birthday. %Your existence and uniqueness brings us joy.}}
\vspace{0.5cm}

\begin{abstract}
We study the symmetrical Dziobek configurations where, in $\mathbb{R}^{d}$, there are $d$ bodies with unit masses at the vertices of a regular $(d-1)$-dimensional simplex of unit edge length and two more bodies with nonzero masses $s,k$ are on the line passing through the center of the simplex and being orthogonal to it. 

In the case of logarithmic potential finiteness is proved for all $s,k\neq 0, d>1$, and we obtain the bifurcation surface in the $(s,k,d)$-space through Gröbner basis computation. Using cylindrical algebraic decompositions, we find $197232$ sample points in the complement of the bifurcation surface. We propose a method to reduce the number to only $202$. By Hermite's root counting theorem, we find that, generically, there can be $0,1,2,3$ or $4$ concave, $1,2,3, $ or $4$ convex, and in totality, $1,2,3,4$ or $5$ such configurations for all dimensions $d>1$. For positive $s$ and $k$, generically, there is a unique convex configuration, while the number of concave ones can be $0,2$ or $4$. All possible combinations for the numbers described above are realized when $d=2$.

We obtain a set of generalized Laura-Andoyer equations equivalent to the central configurations equations for all fixed number of bodies $n=d+h$ and configuration dimension $d$. For homogeneous force law with exponent $a\in \mathbb{R}$, we use the action of permutation group $S_d$ in the Laura-Andoyer equations to reduce the equivalent $\binom{d+2}{2}\binom{d}{2}$ Laura-Andoyer equations to only two generalized polynomial algebraic equations for the studied class of symmetric configurations with two variables representing the positions of the two bodies not at the vertices of the simplex in four parameters $a,d,s,k$.


[2010]{Primary 37N05, 70F10, 70F15, 76B47; Secundary  13P10, 13P15, 14A10. }
\end{abstract}

\section{Introduction}
Central configurations of the $n$-body problem is an important topic. According to Saari, \cite{S} ``Central configurations play a particularly central role in the study of $n$-body systems.'' Enumerating classes of central configurations has been known as a challenging problem. Smale's 6th problem for the 21st century is to prove the finiteness of planar central configurations \cite{SS}.  

As in \cite{A}, we consider the $n$-body problem and central configurations for homogeneous force law with exponent $a\in \mathbb{R}$ (not just for the Newtonian potential of $a=-\frac{3}{2}$) and $n$ particles with nonzero masses (not restricted to be positive) are in $\mathbb{R}^{d}$ for any integer $d>0$. Other references discussing properties for general potentials can be found in \cite{ASV,D,L1,SMPV}. 

When one consider configurations of $n$ bodies in $\mathbb{R}^{d}$, it is natural to restrict $d$ in between $1$ and $n-1$ since $n$ bodies span an affine space of dimension at most $n-1$. Central configurations for $n$ bodies in $\mathbb{R}^{n-1}$ (or, equvialently, for $d+1$ bodies in $\mathbb{R}^{d}$) is well known. It is exactly when $d+1$ bodies are at the vertices of a regular $d$-dimensional simplex for all $a\neq 0$ and for all nonzero masses \cite{A}. On the other hand, central configurations for $d+2$ bodies spanning $\mathbb{R}^{d}$ are far from well known. Such central configurations are called \emph{Dziobek configurations}. Some results about them can be found in \cite{A,L,L1,L2,M,T2}. 

In this paper, we consider Dziobek configurations that contain a sub configuration of a regular simplex. Specifically, we consider the families of symmetrical configurations in $\mathbb{R}^{d}$ where there are $d$ bodies with unit masses at the vertices of a regular $(d-1)$-dimensional simplex of unit edge length, and two more bodies with masses $s,k$ on the line passing through the center of the simplex and being orthogonal to it. Figures $1$ and $2$ show the cases of $d=2,3$ for convex configurations. It is in \cite{L} where Leandro first considered such configurations and called them $(2,d)$-cc's. We follow this terminology here.  

In \cite{L}, Leandro proves the finiteness of $(2,d)$-cc's for $a=-\frac{3}{2}$, all $d>1$, and masses $s,k>0$. For $d=2,3$, he also counts the numbers of such configurations with $s,k>0$. Generically, there are $0,2,4$ concave configurations, while there is always one convex configuration for both $d=2$ and $d=3$. Some exact counts for mass parameters at the bifurcation points are given in \cite{T}.

\begin{figure}
\centering
\begin{minipage}{.50\textwidth}
  \centering
  \includegraphics[width=.72\linewidth]{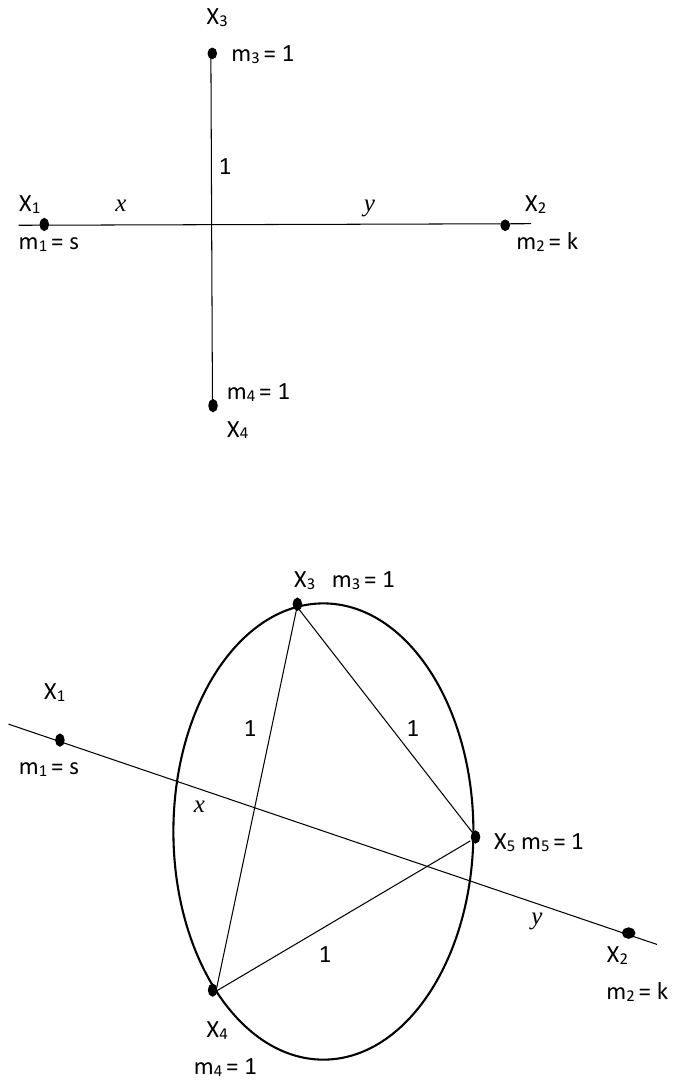}
  \caption{}{Four bodies form a convex symmetric configuration in  $\mathbb{R}^2$}
\end{minipage}%
\begin{minipage}{.50\textwidth}
  \centering
  \includegraphics[width=.72\linewidth]{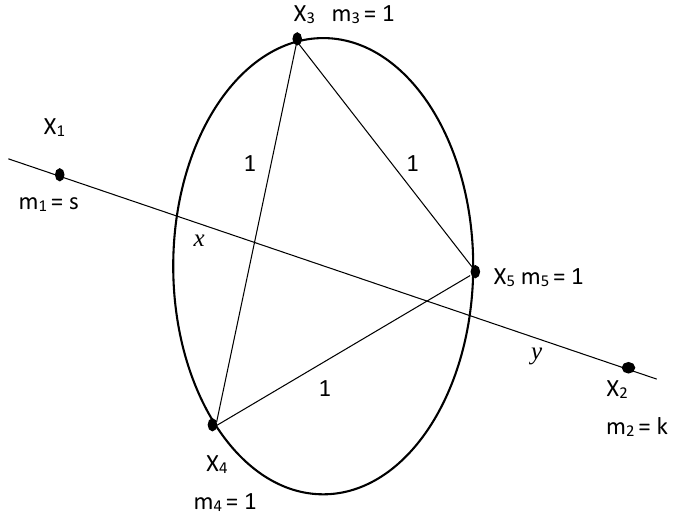}
  \caption{}{Five bodies form a convex symmetric configuration in  $\mathbb{R}^3$}
\end{minipage}
\end{figure}

.

Besides the Newtonian potential of $a=-\frac{3}{2}$, another interesting landmark is when $a=-1$. Planar central configurations for $a=-1$ give special solutions to the $n$ point-vortex problem \cite{O}. Such potential is called the \emph{logarithmic potential} from the defining formula of it. It is believed such potential give a simpler system than that of the Newtonian one and it is made comparisons to the Newtonian case \cite{AC,AC1}. In this paper, we consider the $(2,d)$-cc's for the logarithmic potential not only to gain a better understanding of such Dziobek configurations for all dimensions at a specific $a$ but also in the hope to form conjectures for the Newtonian potential.

Even for the case of $a=-1$, as we will see, there are many challenging and interesting ingredients in enumerating $(2,d)$-cc's for two nonzero masses $s,k$ and for all dimensions $d>1$. Such problem can be reduced to the problem of counting real roots for a system of four polynomial equations in four variables with three parameters. To classify the parameters in terms of the numbers of real roots, we need to find the bifurcation surface in the $(s,k,d)$-space, pick at least one point from all the open connected components separated by it, and count the real roots for the systems at sample parameters.

To rigorously find the bifurcation surface, we employ the symbolic computational method provided in \cite{T}. The computation makes use of Gr\"obner bases and Hermite quadratic forms involving parameters. Vanishing of the determinant of a symmetric matrix for the quadratic form provides a necessary condition for parameters where the number of real zeros changes. Therefore, we can define the bifurcation surface as the zero set of the numerator of the determinant of such matrix, called the \emph{bifurcation polynomial $g$}. 

In our case, the bifurcation polynomial $g$ is a polynomials in $s,k,d$ with degrees $13,13,21$, respectively, with $1607$ terms. It is obtained from computing the determinant of a $12\times 12$ matrix with entries in $\mathbb{Q}(s,k,d)$. With such $g$, we apply the cylindrical algebraic decomposition introduced by Collins \cite{C} that uses exact integer computation to obtain at least one sample point in each open connected component of $g\neq 0$. After about twelve hours of the computation, a list of $197232$ sample points with rational coordinates are obtained. 

We are able to reduce the number $197232$ to only $202$ by proposing an algorithm that refines the list of sample points obtained through cylindrical algebraic decomposition by removing some points belonging to the same component. Our algorithm is based on the depth-first searching for connecting points with some data-preprocessing to employ the divide-and-conquer technique. A program written with Mathematica $11$ implementing our algorithm is provided in \cite{link}. By doing so, we sharpen the upper bound for the number of open connected components to $202$ and also save some time in counting real zeros for those components. This method is one of the main contributions of this paper and it can help solve other problems about counting real solutions of systems of algebraic equations with three parameters.

For those $202$ rational sample points, we again use Gr\"{o}bner bases and Hermite quadratic forms to rigorously count the numbers of real common zeros for integer polynomial systems. We obtained that, generically, there can be $0,1,2,3$ or $4$ concave, $1,2,3,$ or $4$ convex, and in totality $1,2,3,4$ or $5$ $(2,d)$-cc's. For positive $s$ and $k$, generically, there is a unique convex configurations, while the number of concave ones can be $0,2$ or $4$. All possible combinations for the numbers described above are realized in the dimension $d=2$. 

This paper also studies the equivalence between Laura-Andoyer and central configurations equations. G. Meyer derives Laura-Andoyer equations in the planar and spatial cases \cite{ME}. Hagihara proves the equivalence between Laura-Andoyer and central configurations equations in the planar case when the center of mass is at the origin \cite{H}. For proof of this result without assumptions about the center of mass, see \cite{FM} and \cite{B}. Hamptom and Santropete derive a generalization of the Laura-Andoyer equations for central configurations of all dimensions \cite{HS}. We derived generalized Laura-Andoyer equations that are very similar to the ones obtained by Hampton and Santropete. Our contribution is to prove that the generalized  Laura-Andoyer equations are equivalent to the central configuration equations for each choice of the number of bodies and the dimension. As far as we are aware, this is the first equivalence result between Laura-Andoyer equations and central configurations equations with dimension $d>2$. For interesting historical remarks about Laura-Andoyer equations and Dziobek equations, see \cite{A2}

Our paper is divided into two main sections. In section $2$, we derive the algebraic equations involving two equations, two variables, and four parameters and present the enumeration results. Also, we outline our method for computing the bifurcation surface. We give an algorithm to reduce the number of sample points and present the proof of the enumeration results. All of our symbolic computations are carried out in Mathematica $11$. The Mathematica notebook containing all the implementations of our algorithms and the computations can be found in the supplementary material. In section $3$, we derive the generalized Laura-Andoyer equations and discuss the obtaining sets of Laura-Andoyer equations reduced by symmetry. We exemplify it in the case of $(2,d)$-cc's. 
 By using the compatibility of the Laura-Andoyer equations with the symmetry of $(2,d)$-ccs, we reduce the equivalent $\binom{d+2}{2}\binom{d}{2}$ Laura-Andoyer equations to only two algebraic equations with two real variables representing the positions of the two bodies not at the vertices of the simplex in four parameters $a,d,s,k$. Such a technique is different from that in \cite{L} and can be applied to study non-Dziobek classes of symmetrical central configurations.

\section{Enumeration problem and Main Results}

Consider $n$  punctiform  bodies with  $m_{1},...,m_{n} \in \mathbb{R}$ with positions $x_{1},...,x_{n} \in \mathbb{R}^{d}$ and nonzero mutual distances $r_{ij}=\|x_{i}-x_{j}\|$, interacting under a potential of the type

\begin{equation}
U_{a}(x)=\frac{1}{2a+2}\displaystyle \sum_{1\leq i <j \leq n} m_{i}m_{j}r_{ij}^{2a+2}, \text{ if }  a\in \mathbb{R}\setminus\{-1\}, \text{ or}
\end{equation}

\begin{equation}
U_{a}(x)=\displaystyle\sum_{1 \leq i < j \leq n}m_im_j\log r_{ij},
\end{equation}
if $a=-1.$ The equations of motion are given by

\begin{equation}\label{eqN}
m_{i}\ddot{x_{i}}=\displaystyle \sum_{\begin{subarray}{c}j=1\\ j\neq i\end{subarray}}^{n}m_{i}m_{j}(x_{j}-x_{i})r_{ij}^{2a}=\frac{\partial U}{\partial x_{i}},~i=1,...,n.
\end{equation}

By simplicity, define the quantities $R_{ij}=R_{ji}=r_{ij}^{2a}$ and
\begin{equation}\label{eqN3}
\gamma_i= \sum_{\begin{subarray}{c}j=1\\ j\neq i\end{subarray}}^{n}m_{j}R_{ij}(x_j-x_i).
\end{equation}
A central configuration $x$, associated to the potential $U_a$ satisfies the system
\begin{equation}\label{eqc} 
\gamma_i=\lambda(x_i-c),\quad i=1,...,n,
\end{equation}
for which
$$c=\frac{1}{M}\left(m_1x_1+\cdots m_nx_n\right) \qquad \text{and} \qquad M=m_1+\cdots m_n\neq 0$$
are, respectively, the  \emph{center of mass} and the \emph{total masses}.

 Leandro derived the system of a equations a system of algebraic equations for $(2,d)$-cc's with $a=-\frac{3}{2}$ equivalent to  the central configuration equations $\eqref{eqc}$. This system was numbered by $(10)$ in \cite{L}.

 Without loss of generality, for our $d+2$ bodies in $\mathbb{R}^{d}$, we assume $x_{1}=(z,0,\dots,0),x_{2}=(w,0,\dots,0)$ with $z>w$, and $x_{j+2}=(0,\delta_{j})$, where $j=1,\dots,d$ and $\{\delta_{1},\dots,\delta_{d}\}\subset \mathbb{R}^{d-1}$ form a regular $(d-1)$-simplex with center at the origin and unit side length. Let $m_{1}=s, m_{2}=k, m_{j+2}=1$ for $j=1,\dots,d$.
 
 Following Leandro \cite{L}, we obtain the following system of algebraic equations for $(2,d)$-cc's with $a \in \mathbb{R}$:  
 
 \begin{equation}\label{algsys}
  \begin{cases}
   k((z-w)^{2a}-(\frac{d-1}{2d}+w^{2})^{a})(z-w)+((\frac{d-1}{2d}+z^{2})^{a}-1)zd=0,\\
   s((z-w)^{2a}-(\frac{d-1}{2d}+z^{2})^{a})(z-w)-((\frac{d-1}{2d}+w^{2})^{a}-1)wd=0.\\
   \end{cases}
\end{equation}

Let $a=-1$, $r_{1}=z,r_{2}=w,r_{3}=(\frac{d-1}{2d}+z^{2})^{-1},r_{4}=(\frac{d-1}{2d}+w^{2})^{-1}$, the system (\ref{algsys}) becomes a polynomial system (\ref{polygsys}) given below.

\begin{equation}\label{polygsys}
  \begin{cases}
   -k + d r_{1}^2 - d r_{1} r_{2} - d r_{1}^2 r_{3} + d r_{1} r_{2} r_{3} + k r_{1}^2 r_{4} - 2 k r_{1} r_{2} r_{4} + k r_{2}^2 r_{4}=0,\\
   - s + d r_{2}^2 -d r_{1} r_{2}  - d r_{2}^2 r_{4} + d r_{1} r_{2} r_{4}    + sr_{2}^2 r_{3} - 2s r_{1} r_{2} r_{3}  + sr_{1}^2 r_{3}=0,\\
-2 d - r_{3} + d r_{3} + 2 d r_{1}^2 r_{3}=0,\\
-2 d - r_{4} + d r_{4} + 2 d r_{2}^2 r_{4}=0. 
   \end{cases}
\end{equation} 

Recall that the bodies with masses $s$ and $k$ have coordinates $(r_{1},0,\dots,0)$ and $(r_{2},0,\dots,0)$, respectively, in $\mathbb{R}^{d}$, while the remaining $d$ bodies at the vertices of a regular simplex with the center at the origin all have first coordinates of zero. It is easy to see that there is no real zeros for (\eqref{polygsys}) with $r_{1}=r_{2}$ or $r_{1}=0$ or $r_{2}=0$. Note that $r_{1},r_{2}\neq 0$ means that placing the body with mass $s$ or $k$ at the center of the regular simplex cannot form a $(2,d)$-cc.

By symmetry, the numbers of total $(2,d)$-cc's are the numbers of common real zeros of (\ref{polygsys}) with $r_{1}>r_{2}$. It is possible that $r_{1}>r_{2}>0$, $0>r_{1}>r_{2}$, or $r_{1}>0>r_{2}$. The first case corresponds to configurations where the body with mass $k$ is in the convex hull formed by the remaining bodies, the second case corresponds to configurations where the body with mass $s$ is in the convex hull formed by the remaining bodies, and the third case corresponds to configurations where the bodies with mass $s$ and mass $k$ are in the different half regions separated by the hyperplane containing the $d$ bodies at the vertex of the regular simplex. Therefore, concave configurations correspond to real zeros with $r_{1}>r_{2}>0$ or $0>r_{1}>r_{2}$, while convex configurations correspond to real zeros with $r_{1}>0>r_{2}$.  

Here are the main results of the paper.
\begin{theorem} \label{maint}
For all non-zero real numbers $s,k$ and integers $d>1$, there are at most $12$ complex zeros for system (\ref{polygsys}). There are generically $0,1,2,3,4$ real zeros with $r_{1}>r_{2}>0$ or $0>r_{1}>r_{2}$, and $1,2,3,4$ real zeros with $r_{1}>0>r_{2}$. The total numbers of real zeros with $r_{1}>r_{2}$ are generically $1,2,3,4,5$. For $s,k>0$, there are generically $0,2,4$ real zeros with $r_{1}>r_{2}>0$ or $0>r_{1}>r_{2}$, and one zero with $r_{1}>0>r_{2}$. All the possibilities are realized for $d=2$.
\end{theorem}

\begin{corollary}\label{main}
For $a=-1$, the numbers of $(2,d)$-cc's are finite with an upper bound $12$. Generically, there are $0,1,2,3,4$ concave, $1,2,3,4$ convex, and in totality $1,2,3,4,5$ such central configurations. For positive masses, there is a unique convex $(2,d)$-cc and there are $0,2,4$ concave $(2,d)$-cc's, generically. No new information on the numbers of $(2,d)$-cc's can be found besides what is known for $d=2$. In Table $1$, a list of parameters that realize all possible numbers of $(2,d)$-cc's is presented.
\renewcommand{\arraystretch}{2}
\begin{table}[h]
\begin{center}
\begin{tabular}{|c|c|c|c|c|}
\hline
$d$ & $(s,k)$& concave & convex & total  \\ \hline 
$2$ & $(9,\frac{3}{2})$&$0$&$1$& $1$ \\ \hline
$2$ & $(-\frac{15}{8},-\frac{5}{8})$&$0$& $2$&$2$ \\ \hline
$2$ & $(-\frac{15}{18},-\frac{1}{4})$&$0$& $3$&$3$ \\ \hline
$2$ & $(-\frac{15}{8},-\frac{15}{8})$&$0$& $4$&$4$ \\ \hline
$2$ & $(-\frac{67}{128},-18)$&$1$& $1$&$2$ \\ \hline
$2$ & $(-\frac{1}{4},-5)$&$1$& $2$&$3$ \\ \hline
$2$ & $(-\frac{15}{8},-3)$&$1$& $3$&$4$ \\ \hline
$2$ & $(\frac{1}{4},2)$&$2$& $1$&$3$ \\ \hline
$2$ & $(\frac{1}{4},-\frac{97}{384})$&$2$& $2$&$4$ \\ \hline
$2$ & $(-\frac{11}{5},\frac{71}{32})$&$3$& $1$&$4$ \\ \hline
$2$ & $(-\frac{1}{4},1)$&$3$& $2$&$5$ \\ \hline
$2$ & $(\frac{1}{4},1)$&$4$& $1$&$5$ \\ \hline
\end{tabular}
\caption{All different numbers of $(2,d)$-cc's and $12$ examples}
\end{center}
\end{table}\renewcommand{\arraystretch}{2}
\end{corollary}

\subsection{Obtaining the bifurcation surface in $\mathbb{R}^{3}$}
In this subsection, we focus on obtaining the \emph{bifurcation surface} in $\mathbb{R}^{3}$ containing parameters where the numbers of real zeros may change. The main tools are Gr\"{o}bner bases \cite{CLO} and Hermite quadratic forms \cite{BPR}.

Let $f_{1},f_{2},f_{3},f_{4}$ denote the polynomials in (\ref{polygsys}). We begin with finding a \emph{generic Gr\"{o}bner basis} for our polynomial system $\mathcal{F}:=\{f_{1},f_{2},f_{3},f_{4}\}$ in $\mathbb{Z}\left[s,k,d\right]\left[ r_{1},r_{2},r_{3},r_{4} \right]$. In other words, if $\mathcal{F}^{s,k,d}$ denotes the real polynomial system after specializing $\mathcal{F}$ at $(s,k,d)\in \mathbb{R}^{3}$, then we will find a set of polynomial $\mathcal{G} \subset \mathbb{Z}\left[s,k,d\right]\left[ r_{1},r_{2},r_{3},r_{4} \right]$ such that $\mathcal{G}^{s,k,d}$ is a Gr\"{o}bner basis for $\mathcal{F}^{s,k,d}$ for all points $(s,k,d)\in \mathbb{R}^{3}$, except for those in the zero set of a polynomial. 

\begin{lemma}\label{h}
The system $\mathcal{F}=\{f_{1},f_{2},f_{3},f_{4}\}$ has a generic Gr\"{o}bner basis $\mathcal{G}=\{g_{1},\dots,g_{11}\}$ satisfying that $\mathcal{G}^{s,k,d}$ is a Gr\"{o}bner basis for $\mathcal{F}^{s,k,d}$ for all $(s,k,d)$ with $q\neq 0$, where 
\[
\begin{aligned}
q:=&(d-1) d k (d + k) (d + 2 k-1) s (d + s)(d + 
   2 s-1)(k + s)\times \\
    &(d^2 - 4 k s) (d k + d s + 2 k s)(d^3 + d^2 k + 2 d^2 s - 
   2 d k s - 4 k s^2).
\end{aligned}
\]
\end{lemma}

\begin{proof}
We first consider $\mathcal{F}$ as a system in $\mathbb{Q}(s,k,d)\left[ r_{1},r_{2},r_{3},r_{4} \right]$, where $\mathbb{Q}(s,k,d)$ is the field of rational functions in $s,k,d$. Applying the \emph{GroebnerBasis} command in Mathematica with the graded reverse lexicographic order on the variables $r_{1},r_{2},r_{3},r_{4}$, we obtain the semi-reduced Gr\"{o}bner basis $\mathcal{G}=\{g_{1},\dots,g_{11}\}$ obtained from clearing denominators for the reduced Gr\"{o}bner basis in $\mathbb{Q}(s,k,d)$ \cite{link1}. Let $lcm$ be the least common multiple of all leading coefficients in $\mathbb{Z}\left[ s,k,d \right]$ of all polynomials in $\mathcal{G} \cup \mathcal{F}$.

Next, we compute a Gr\"{o}bner basis for $\{f_{1},f_{2},f_{3},f_{4}\}$ viewing  
$r_{1},r_{2},r_{3},r_{4}$ and $s,k,d$ as variables with the graded reverse lexicographic order to show $(lcm)g_{i}$ is in the ideal generated by $\{f_{1},f_{2},f_{3},f_{4}\}$ in $\mathbb{Q}\left[ r_{1},r_{2},r_{3},r_{4},s,k,d \right]$ for all $i=1,\dots,11$. By doing so, for all $(s,k,d)$ with $lcm \neq 0$, we are guaranteed that $\mathcal{G}^{s,k,d}$ will be a Gr\"{o}bner basis for $\mathcal{F}^{s,k,d}$. The procedure of obtaining such a generic Gr\"{o}bner basis follows from exercise $9$ of Chapter $6.3$ in \cite{CLO}. Here $q$ is the square-free part of $lcm$.
\end{proof}

Now, we are ready to give definitions for some terms used in this paper.
\begin{definition}
\begin{enumerate}
  \item The set $\mathcal{P}=(\mathbb{R}^{2}\times\{d>1\})\setminus \{q=0\}$, where $q$ is given in Lemma \ref{h}, is the \emph{parameter space} for our generic results in this paper.
  \item The \emph{bifurcation set} $\mathcal{B}$ is defined as the projection of the common zeros in $\mathbb{R}^{7}$ of $\{f_{1},f_{2},f_{3},f_{4},J\} \subset \mathbb{Q}\left[ r_{1},r_{2},r_{3},r_{4},s,k,d \right]$, where $J$ is the Jacobian determinate of $f_{i}$'s with respect to $r_{j}$'s, onto the parameter space $\mathcal{P}$.
  \item A polynomial in $\mathbb{Z}\left[ s,k,d \right]$ is called a \emph{bifurcation polynomial} if its zero set, called the \emph{bifurcation surface}, in $\mathbb{R}^{3}$ contains $\mathcal{B}$. 
\end{enumerate}
\end{definition}

\begin{lemma}\label{g}
A bifurcation polynomial $g \in \mathbb{Z}\left[ s,k,d \right]$ that has degrees $13,13,21$ in $s,k,d$, respectively, and $1607$ terms can be computed from $\mathcal{F}$ without using its Jacobian determinate $J$. 
\end{lemma}

\begin{proof}
The set $\mathcal{G}$ computed in the proof of Lemma \ref{h} is a Gr\"{o}bner basis of $\mathcal{F}$ in $\mathbb{Q}(s,k,d)\left[ r_{1},r_{2},r_{3},r_{4} \right]$. Therefore, we can use it to compute a symmetric matrix $\mathcal{H}$ of the Hermite quadratic form in $\mathbb{Q}(s,k,d)$ whose dimension is the number of zeros in the algebraic closure of $\mathbb{Q}(s,k,d)$ counted with multiplicities, and whose rank is the number of distinct zeros in the same field \cite{BPR}. For real triples $(s,k,d)$ in $\mathcal{P}$, the rank of $\mathcal{H}^{s,k,d}$ gives the number of distinct complex zeros of $\mathcal{F}^{s,k,d}$. If $(s,k,d)$ is in $\mathcal{B}$, then there is a real zero of $\mathcal{F}^{s,k,d}$ such that the multiplicity is greater than one. Therefore, $\mathcal{H}^{s,k,d}$ cannot have full rank. (Full rank for $\mathcal{H}$ means all the multiplicities are one.) So, the determinant of $\mathcal{H}^{s,k,d}$ must be zero. After factoring the determinant of $\mathcal{H}$ in $\mathbb{Q}(s,k,d)$ and removing factors in $lcm$ obtained in the proof of Lemma \ref{h}, we obtain our bifurcation polynomial $g$.
\end{proof}

Figures $3$ and $4$ show the curves determined by $g(s,k,d)=0$ for $d=2$ and $d=3$, respectively, in the first quadrant of the $sk$-plane. Note the differences between our figures for $a=-1$ and the figures for $a=-\frac{3}{2}$ in \cite{L}. There are four open connected components determined by the complement of the bifurcation curves in \cite{L}, while there are nine components in our cases. Also, there is only one singular point for $a=-\frac{3}{2}$ on the line $s=k$, while, in our cases, there are four, two on the line $s=k$ and two not on it. 

\begin{figure}[h!]
\centering
\begin{minipage}{.52\textwidth}
  \centering
  \includegraphics[width=.72\linewidth]{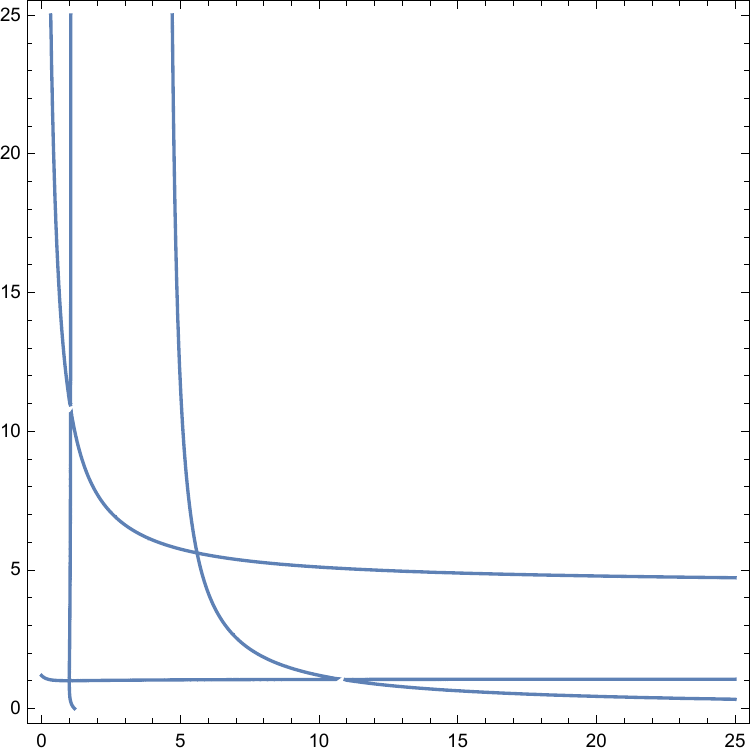}
  \caption{}{The curve of g(s,k,d)=0 for d=2 in the first quadrant of the sk-plane.}
\end{minipage}%
\begin{minipage}{.52\textwidth}
  \centering
  \includegraphics[width=.72\linewidth]{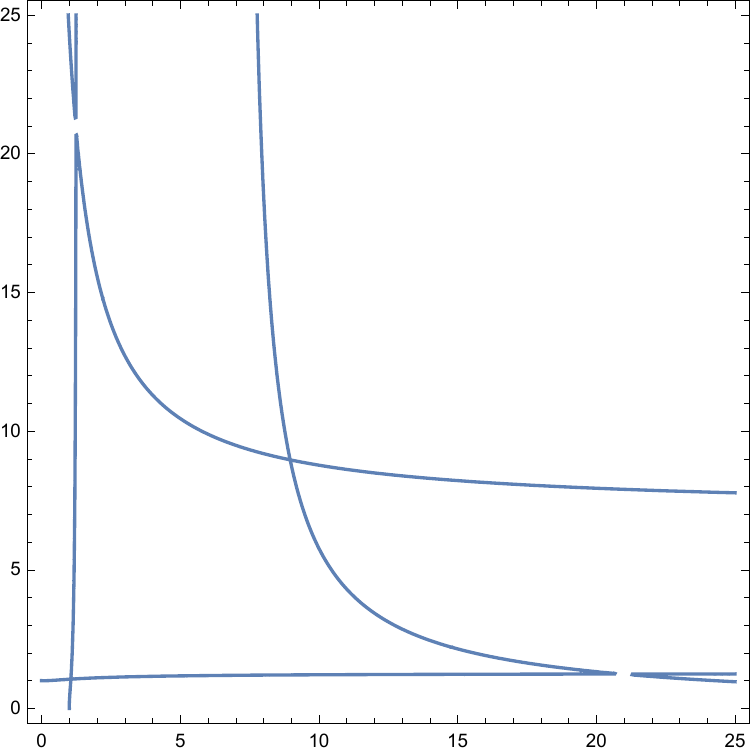}
  \caption{}{The curve of g(s,k,d)=0 for d=3 in the first quadrant of the sk-plane.}
\end{minipage}
\end{figure}

Note the parameter space $\mathcal{P}$ is partitioned by the zero set of $q$ into finite union of open connected components. Some of these components are further partitioned by the zero set of $g$. In each component, the number of real zeros is a constant. Also, it is easy to see that for all $(s,k,d)\in \mathcal{P}$, there is no zeros for $r_{1}=0$ or $r_{2}=0$ or $r_{1}=r_{2}$. Therefore, the numbers of concave configurations which correspond to zeros with $r_{1}>r_{2}>0$ or $0>r_{1}>r_{2}$ and convex configurations which correspond to zeros with $r_{1}>0>r_{2}$ are constants in each of the open connected components.

\subsection{Open connected components in the parameter space with $g\neq0$}\label{openc}
Here, our goal is to pick at least one sample points from each of the open connected components in the parameter space where $g\neq0$. We have the following first estimation on the upper bound of the number of such components. 

\begin{proposition}\label{197323}
The number of open connected components in $\mathcal{P} \setminus \{g=0\}$ is at most $197232$. 
\end{proposition}

\begin{proof}
Recall the parameter space is $\mathcal{P}=(\mathbb{R}^{2}\times\{d>1\})\setminus \{q=0\}$. The subset in $\mathbb{R}^{3}$ with $g\neq0, q\neq0, d>1$ is a semialgebraic set which is a finite union of open connected components \cite{BPR}. (As mentioned in the previous subsection, the number of real zeros is a constant in each component.) To find at least a sample point from each component, we use the tool of cylindrical algebraic decomposition(CAD) introduced by Collins \cite{C}. In Mathematica, the command \emph{SemialgebraicComponentInstances} is based on the CAD algorithm and gives at least a sample point from each component of a semialgebraic set \cite{link2}.  

In our case, the computation is non-trivial. We apply \emph{SemialgebraicComponentInstances} for $\{ gq>0, d>1\}$ and $\{ gq<0, d>1\}$ and obtain $100771$ and $96461$ rational sample points after about $12$ hours of the computation. At least one point from each of the open connected components of $\mathcal{P}\setminus \{g=0\}$ is picked. There are totally $197232$ of them, which gives an upper bound for the number of open connected components in $\mathcal{P} \setminus \{g=0\}$.
\end{proof}

The number of sample points obtained from \emph{SemialgebraicComponentInstances} is far from minimal. The number of it gives only a rough upper bound for the number of open connected components of a semialgebraic set. To reduce the number of sample points, we can connect those in the same component by continuous arcs and remove some belonging to the same component. One choices of the arcs is a union of finite line segments. Using more than one line segments can be a good approach when working in $\mathbb{R}^{2}$ where one may use numerical plots as hints about where the intermediate points should be. Considering we are working in $\mathbb{R}^{3}$ and that the number in our case is $197232$, we only use a single line segment.

Given a list of sample points of the semialgebraic set $f\neq 0$, we connect points as follows. We randomly choose a point $p_{1}$ from the list and then find all its neighbours defined as all other sample points connecting to it with a single line segment. This can be achieved by applying the \emph{CountRoots} command in  Mathematica \cite{CK}. It counts the number of real zeros rigorously in any closed interval for univariate polynomials. So, we can parametrize the line segment in the $(s,k,d)$-space with $t\in \left[ 0,1\right]$, where $t=0,1$ give two sample points, and composite with $f$ to see if there is a real zero for the composition in $(0,1)$. If not, two points are connected by the line segment. 

We have $197232$ sample points to classify in open components, In order to do this task  we create a procedure named  \emph{ConnectedCom}, described in the following steps. We choose a simple point $p_1$ and we collect all neighbors of $p_1$. After this step, we continue finding neighbours of the neighbours and stop until no more points can be included in the family of $p_{1}$. We repeat the process for the rest of the sample points. Eventually, we partition the list into a union of families of points, where all points in the same family belong to the same component. Choosing one representative point from each family, we obtain a potentially shorter list of sample points. 

Using the procedure \emph{ConnectedCom} and some tricks on the set of $197232$ sample points, we are able to have a much better estimations on the numbers of open connected components.

\begin{proposition}\label{217}
The number of open connected components in $\mathcal{P} \setminus \{g=0\}$ is at least $78$ and at most $217$. The number of open connected components in $(\{s>0\}\times\{k>0\}\times\{d>1\})\setminus \{gq=0\}$ is at least $6$ and at most $43$.
\end{proposition}

\begin{proof}
It may not be a good idea to apply \emph{ConnectedCom} directly to a list of $197232$ points. We employ some divide and conquer technique by dividing the original set into many smaller subsets and apply \emph{ConnectedCom} to each subset separately. By doing so, parallel computing is possible. In our case, we partition the set of $197232$ points into a union of $78$ subsets. For each subset, we employ further divide and conquer.

Next, we describe how those $78$ subsets are obtained and what further divide and conquer can be done in each subset. Recall $\mathcal{P}=(\mathbb{R}^{2}\times\{d>1\})\setminus \{q=0\}$. Note among the $197232$ points, $100771$ of them are sample points for $\{ gq>0, d>1\}$ in $\mathcal{P}$ and $96461$ of them are sample points for $\{ gq<0, d>1\}$ in $\mathcal{P}$. There are $12$ factors in $q$, two of them are $(d-1)$ and $d$, which can be ignored for we require $d>1$. For the remaining $10$ factors, there are $2^{10}$ different combinations of signs. From the $100771$ sample points, only $40$ combinations of the signs are realized, while only $38$ combinations of the signs are realized for the $96461$ sample points. 

We also find that there are $31$ same combinations of the signs both in the $40$ ones and the $38$ ones. And, there are $47$ combinations of the signs in the $40$ ones or the $38$ ones. It means that, among the $2^{10}$ different combinations of the signs from the $10$ factors in $q$, exactly $47$ of them are realized. Our parameter space $\mathcal{P}$ can be partitioned into a union of $47$ open components. Each component contains points with the same combination of the $10$ signs for factors in $q$. There are $31$ of them containing points with $g>0$ and points $g<0$. The rest $16$ components contain only points of the same sign for $g$. 

Since we consider the open connected components of $\mathcal{P}\setminus \{g=0\}$, each of the $31$ open components are further decomposed into two open components. Therefore, we need $78$ open components for $\mathcal{P}\setminus \{g=0\}$. Note $78=16+2(31)$. Each of the $78$ open component contains points with the same combination of the $10$ signs for factors in $q$ and the sign for $g$. Among them, $6$ components contain all sample points with positive $s$ and $k$. Since sample points with different combinations of the signs cannot be in the same open connected component, the number $78$ provides a lower bound for the total number of the open connected components in $\mathcal{P}\setminus \{g=0\}$. Similarly, the number $6$ is a lower bound for the  total number of the open connected components in $(\{s>0\}\times\{k>0\}\times\{d>1\})\setminus \{gq=0\}$. 

After partitioning the set of $197232$ points into the union of $78$ smaller subsets, in each subset $S_{0}$, we reduce the number of sample points as follows. We first collect points with the same $s$ coordinates into $n_{1}$ groups. We then apply \emph{ConnectedCom} to all those groups separately. With a possible refined subset, $S_{1}$, obtained from removing redundant points, we collect points with the same $k$ coordinates into $n_{2}$ groups. Applying \emph{ConnectedCom} to those groups, separately, and removing redundant points, we again obtain another possible refined subset, $S_{2}$. Collecting points with the same $d$ coordinates into $n_{3}$ groups, and applying \emph{ConnectedCom} to those groups, separately, we obtain $S_{3}$. Finally, we apply \emph{ConnectedCom} on $S_{3}$ to obtain $S_{4}$. Let $\#S$ denote the cardinality of the set $S$. We have $\#S_{i-1}\geqslant \#S_{i}\geqslant n_{i}$ for $i=1,2,3,4$, where $n_{4}=1$. We totally apply \emph{ConnectedCom} $(n_{1}n_{2}n_{3}+1)$ times. 

For example, let's consider the subset containing the sample point $(s,k,d)=(-80,-\frac{21}{2},\frac{1289}{16})$. For such subset, $\#S_{0}=4384$. In our reductions in four stages, we have $n_{1}=1511$, $\#S_{1}=1561$, $n_{2}=357$, $\#S_{2}=372$, $n_{3}=160$, $\#S_{3}=176$, and, finally, $\#S_{4}=1$. In fact, among the $78$ subsets, $\#S_{0}$'s are in between $42$ to $7458$ and there are $34$ of $\#S_{0}>2500$. On the other hand, $\#S_{4}$'s are in between $1$ to $13$, and there are $48$ of $\#S_{4}=1$ and $65$ of $\#S_{4}\leqslant 5$. Recall that we have $16$ open components with the same combination of the $10$ signs in $q$ containing only points of the same sign for $g$. Among these $16$ components, there are $12$ of $\#S_{4}=1$. On the other hand, we have $31$ open components with the same combination of the $10$ signs in $q$ containing points with $g>0$ and points with $g<0$. From the information of $\#S_{4}$'s, there are $16$ of them being proved to have just $2$ open connected components.

In summary, the set of $100771$ sample points with $gq>0,d>1$ partitioned into a union of $40$ subsets is reduced to a set of $99$ sample points. Among the $40$ subsets, $3$ of them contain all points with positive $s$ and $k$, and the union of them is reduced to a set of only $16$ sample points. On the other hand, the set of $96461$ sample points with $gq<0,d>1$ partitioned into a union of $38$ subsets are reduced to a set of $118$ sample points. Among the $38$ subsets, $3$ of them contain all points with positive $s$ and $k$, and the union of them is reduced to a set of only $27$ sample points. 
 
We make a final remark on applying \emph{ConnectedCom} on each of the $78$ subsets. Although we apply \emph{SemialgebraicComponentInstances} on $\{ gq>0, d>1\}$ and $\{ gq<0, d>1\}$, when we try to test if two points are connected by a line segment, we only need to test if the line pass through $gw=0$, where $w=(d^2 - 4 k s) (d k + d s + 2 k s)(d^3 + d^2 k + 2 d^2 s -  2 d k s - 4 k s^2)$ is the non-linear factors in $q$. In each of the subset $S_{0}$, the signs of all the factors are fixed. For each linear factor $l$, it determines two half spaces. If two points have the same signs, they must belong to the same half space and, therefore, the line segment determined by them will never pass $l=0$.   
\end{proof}

In fact, in subsection $\ref{proof}$, when we prove our generic results, we will show there are $15$ points that can be further removed, since the open connected components containing them in the region of $d<2$. Therefore, we can eventually reduce the number of sample points to $202$ from a set of $197232$ sample points. 

\subsection{Proof of theorem $1$}\label{proof}
In this subsection, we will prove Theorem \ref{maint} in the following subsubsections. In the first one, we prove the finiteness results. In the second one, we prove generic results for $d=2$. Finally, in the third subsubsection, we prove generic results for all $d>1$. Corollary \ref{main} follows directly from Theorem \ref{maint} and the results in the subsubsection. 

\subsubsection{For the finiteness results}
Recall that, in Lemma \ref{h}, we compute a Gr\"{o}bner basis $\mathcal{G}=\{g_{1},\dots,g_{11}\}$ for $\mathcal{F}=\{f_{1},f_{2},f_{3},f_{4}\}$ over the field $\mathbb{Q}(s,k,d)$. Also, by Lemma \ref{h}, we know that $\mathcal{G}^{s,k,d}$ is a Gr\"{o}bner basis of $\mathcal{F}^{s,k,d}$ for $q(s,k,d)\neq0$. For those $(s,k,d) \in \mathcal{P}$, the leading power products of $r_{i}$'s are $ r_{1}r_{3},r_{2}^2,r_{1}r_{2},r_{1}^2,r_{4}^3,r_{3}r_{4}^2,r_{1}r_{4}^2,r_{3}^{2}r_{4},r_{2}r_{3}r_{4}, r_{3}^3, \\ r_{2}r_{3}^2 $. In particular, the power products $r_{2}^2,r_{1}^2,r_{4}^3,r_{3}^3$ prove the finiteness. The number of power products not divided by any of the $11$ power products is $12$, which gives an upper bound for the number of complex zeros \cite{CLO}. 

Next, we consider those $(s,k) \in \mathbb{R}^{2}$ with $s,k\neq 0$ and $d\in \mathbb{N}$ with $d>1$ such that $q(s,k,d)=0$. They are zeros of $(d + k)(d + 2 k-1)(d + s)(d + 2 s-1)(k + s)(d^2 - 4 k s)(d k + d s + 2 k s)(d^3 + d^2 k + 2 d^2 s - 2 d k s - 4 k s^2)$ that is the product of eight factors of $q$. All the factors are linear polynomials in either $s$ or $k$. We treat each factor separately. By solving $s$ or $k$ in terms of the other parameters, we can reduce the number of parameters to two and obtain eight systems of polynomial equations in only two parameters. 

For each one of the eight systems, using the similar technique in computing generic Gr\"{o}bner bases \cite{T} by a block order where $r_{1},r_{2},r_{3},r_{4}$ are in the graded reverse lexicographic order and two parameters also treated as variables are in a lexicographic order, we obtain a polynomial $u$ in two variables. Now, parameters with $u\neq 0$ are the generic cases where the finiteness and upper bounds of the complex zeros can be obtained as we did for $q\neq 0$. The upper bounds can be $10$ or $12$.

Again, the parameters with $u=0$ need to be treated differently. For each factor of $u$, we compute a Gr\"{o}bner basis for a system of $5$ equations ($4$ from $f_{1},f_{2},f_{3},f_{4}$ and one factor of $u$) in $6$ variables ($r_{1},r_{2},r_{3},r_{4}$, a variable $t \in \{s,k\}$, and $d$) with the lexicographic order where $r_{4}>r_{3}>r_{2}>r_{1}>t>d$. In those Gr\"{o}bner bases, we can always find a polynomial in only the variables of $r_{1}$ and $t,d$ where the degree in $r_{1}$ is $4,6,8$, or $10$, and the \emph{leading coefficients} are integral polynomials only in $d$. Also, we can always find another polynomial in $r_{1},r_{2}$ and $t,d$ where the degree of $r_{2}$ is one. From these, we prove the finiteness of the complex zero with varies upper bounds $4,6,8$, or $10$.

\subsubsection{For the case of $d=2$}
Here we summarize our generic results for $d=2$ in the following proposition.
\begin{proposition}\label{d2}
Let $d=2$. There are generically $0,1,2$ real zeros with $r_{1}>r_{2}>0$, $0,1,2$ real zeros with $0>r_{1}>r_{2}$, and $1,2,3,4$ real zeros with $r_{1}>0>r_{2}$. The total numbers of real zeros with $r_{1}>r_{2}>0$ or $0>r_{1}>r_{2}$ are generically $0,1,2,3,4$. The total numbers of real zeros with $r_{1}>r_{2}$ are generically $1,2,3,4,5$. 

For $s,k>0$, we have the following enumerations. There are generically $0,2$ real zeros with $r_{1}>r_{2}>0$, $0,2$ real zeros with $0>r_{1}>r_{2}$, and one zero with $r_{1}>0>r_{2}$. The total numbers of real zeros with $r_{1}>r_{2}>0$ or $0>r_{1}>r_{2}$ are generically $0,2,4$. The total numbers of real zeros with $r_{1}>r_{2}$ are generically $1,3,5$. 
\end{proposition}

\begin{proof}
Following the same procedure described in subsection $\ref{openc}$, we obtain $93$ sample points from applying \emph{SemialgebraicComponentInstances} for computing sample points and \emph{ConnectedCom} for the reduction of the number of them. There are $15$ sample points with positive $s$ and $k$.

Next, we describe how to count the numbers of real zeros of an integer polynomial system for the cases of $r_{1}>r_{2}>0$, $0>r_{1}>r_{2}$, and $r_{1}>0>r_{2}$. Our method is based on Hermite's root counting theorem. Given a polynomial $v$, the signature of the matrix obtained from the Hermite's quadratic form with respect to $v$ is the number of real zeros for $v>0$ minus that of $v<0$ \cite{BPR}. In particular, when $v=1$, the signature gives the number of distinct real zeros. Using $v_{1}=1,v_{2}=r_{1}-r_{2},v_{3}=r_{1},v_{4}=r_{2},v_{5}=(r_{1}-r_{2})r_{1},v_{6}=(r_{1}-r_{2})r_{2},v_{7}=r_{1}r_{2},v_{8}=(r_{1}-r_{2})r_{1}r_{2}$, we obtain numbers of real zeros for $r_{1}>r_{2}>0$, $0>r_{1}>r_{2}$ and $r_{1}>0>r_{2}$. Computing $93$ systems, we obtain our generic results. Among those $93$ cases, $15$ of them correspond to positive $s,k$. From those systems, we obtain generic results for positive $s,k$. 

Let $m$ be the number of real zeros with $r_{1}>r_{2}>0$ or $0>r_{1}>r_{2}$, and $n$ be the number of real zeros with $r_{1}>0>r_{2}$. Then $(m,n)$ can only be $(0,1),(0,2),(0,3),(0,4),(1,1),(1,2),(1,3),(2,1),(2,2),(3,1),(3,2)$, or $(4,1)$. In Table \ref{d2table}, we show the numbers of sample points in four quadrants in the $sk$-plane that give each of the $12$ $(m,n)$'s and pick a representative sample point if such number is positive. In particular, it is clear to see from Table \ref{d2table} which of the pairs $(m,n)$ are possible for all $(s,k)$ in the same quadrant. 

\renewcommand{\arraystretch}{2}
\begin{table}[h]
\begin{center}
\begin{tabular}{|c|c|c|c|c|c|}
\hline
$(m,n)$ & $s,k>0$ & $s<0,k>0$ &$s,k<0$& $s>0,k<0$   \\ \hline 
$(0,1)$ & $5,(9,\frac{3}{2})$ & $3,(-4,3)$ & $0$ & $2,(\frac{9}{8},-11)$\\ \hline
$(0,2)$ & $0$ & $0$ & $5,(-\frac{15}{8},-\frac{5}{8})$ & $0$\\ \hline
$(0,3)$ & $0$ & $0$ & $4,(-\frac{15}{8},-\frac{1}{4})$ & $0$\\ \hline
$(0,4)$ & $0$ & $0$ & $2,(-\frac{15}{8},-\frac{15}{8})$ & $0$\\ \hline
$(1,1)$ & $0$ & $1,(-\frac{177}{16},12)$ & $15,(-\frac{67}{128},-18)$ & $1,(\frac{9}{4},-\frac{71}{32})$\\ \hline
$(1,2)$ & $0$ & $5,(-\frac{1}{4},6)$ & $7,(-\frac{1}{4},-5)$ & $3,(\frac{1}{4},-1)$\\ \hline
$(1,3)$ & $0$ & $0$ & $9,(-\frac{15}{8},-3)$ & $0$\\ \hline
$(2,1)$ & $9,(\frac{1}{4},2)$ & $2,(-\frac{177}{16},1)$ & $0$ & $3,(\frac{1}{4},-3)$\\ \hline
$(2,2)$ & $0$ & $5,(-\frac{1}{4},\frac{1}{6})$ & $1,(-\frac{177}{16},-3)$ & $4,(\frac{1}{4},-\frac{97}{384})$\\ \hline
$(3,1)$ & $0$ & $1,(-\frac{11}{5},\frac{71}{32})$ & $0$ & $1,(\frac{17}{8},-\frac{25}{12})$\\ \hline
$(3,2)$ & $0$ & $2,(-\frac{1}{4},1)$ & $0$ & $2,(\frac{1}{4},-\frac{7}{32})$\\ \hline
$(4,1)$ & $1,(\frac{1}{4},1)$ & $0$ & $0$ & $0$\\ \hline
\end{tabular}
\caption{Numbers of sample points in four quadrants giving $12$ pairs of $(m,n)$ and examples for the cases of posive numbers}\label{d2table}
\end{center}
\end{table}\renewcommand{\arraystretch}{2}
\end{proof}

\subsection{For the general case of $d>1$}
Here, we first reduce the number of $217$ to $202$ for sample points in the general case of $d>1$. Note that, among the $217$ sample points, we find $15$ sample points where the systems obtained from evaluating at them for (\ref{polygsys}) either have $4$ real common zeros with $r_{1}>r_{2}>0$ (for $9$ sample points) or have $4$ real common zeros with $0>r_{1}>r_{2}$ (for $6$ sample points). They all have $d$-coordinates slightly greater than $1$ and strictly smaller than $2$. We will claim that all the open connected components they belong to are contained in the region of $d<2$. 

In fact, by Proposition \ref{d2}, we know, for $d=2$, the number of real common zeros with $r_{1}>r_{2}>0$ is $0,1$ or $2$ and that with $0>r_{1}>r_{2}$ is also $0,1$ or $2$. Therefore, all the open connected components containing those $15$ sample points do not intersect with the plane $d=2$. Since all those $15$ sample points are in the region of $d<2$, all these components must also be in the region of $d<2$. Therefore, those $15$ points can be removed from our list of $217$.  

For the remaining $202$ sample points, we again use the same method as in the proof of Proposition \ref{d2} to count the numbers of real zero with $r_{1}>r_{2}>0$, $0>r_{1}>r_{2}$, and $r_{1}>0>r_{2}$. Again, let $m,n$ be defined as in the proof of Proposition \ref{d2}. Each one of the $202$ sample points determines a pair of $(m,n)$ that belongs to one of the $12$ cases for $d=2$. Table \ref{dtable} shows the numbers of sample points in each of the $12$ cases. For those $43$ sample points with positive $s$ and $k$ among the $202$ ones, there are $8,19,$ and $16$ of them in the cases of $(m,n)=(0,1),(2,1),$ and $(4,1)$, respectively. 

\renewcommand{\arraystretch}{2}
\begin{table}[h]
\begin{center}
\begin{tabular}{|c|c|c|c|c|c|}
\hline
$(0,1)$ & $(0,2)$ & $(0,3)$ &$(0,4)$& $(1,1)$ & $(1,2)$ \\ \hline  
$13$ & $17$ & $10$ &$3$& $59$ & $31$   \\ \hline 
$(1,3)$ & $(2,1)$ & $(2,2)$ & $(3,1)$ & $(3,2)$ & $(4,1)$ \\ \hline
$9$ & $27$ & $11$ & $2$ & $4$ & $16$ \\ \hline 
\end{tabular}
\caption{Numbers of sample points giving $12$ paris of $(m,n)$}\label{dtable}
\end{center}
\end{table}\renewcommand{\arraystretch}{2}

\section{Generalized Laura-Andoyer Equations }

In this section we derive a set of Laura-Andoyer equations equivalent to the the central configuration equations for arbitrary dimension and number of bodies. Next we use Laura-Andoyer equations and Dziobek equations to obtain equations \eqref{algsys} (equations (10) in \cite{L} for $a=\frac{-3}{2}$. Here, we use a  approach different from the Leandro's one. We derive Laura-Andoyer equations for Dziobek configurations first from $(5)$. We find there are only two equations for all $d$ and use them to obtain our algebraic equations involving parameters $s,k,d,a$. In our approach, the hole of the symmetry group of the $(2,d)-$cc's in the obtaining of the  equations become clear.

Let $x=(x_1,...,x_{d+h})$ be a central configuration of dimension $d$. We can suppose without loss of generality that each $x_{i}\in \mathbb{R}^{d}$. The  \emph{matrix of the configuration} $x$ is the $(d+h)\times (d+h)$ given by

$$X=\left(\begin{array}{ccc}
    1 & \ldots & 1 \\
    x_{11} & \ldots & x_{1(d+h)} \\
    \vdots &\ddots& \vdots\\
    x_{d1} & \ldots & x_{d(d+h)}\\
    0 & \cdots &0\\
    \vdots &\ddots& \vdots\\
    0 & \cdots &0
    \end{array}\right).$$
    
 Since the rank of $X$ is $d+1$, we have that $|X_{l_{1}...l_{h-1}}|\neq0$, for some choice of $l_1,...,l_{h-1}$.

 Denote by $x_{l_1\cdots l_{h}}$, where $1 \leq l_1<l_2<\cdots<l_h \leq n$, the subconfiguration of $d$ bodies obtained from the configuration $x$ by removing the bodies $x_{l_1},...,x_{l_h}$. The $d\times d$ configuration matrix of $x_{l_1...l_h}$ will be denoted by
$X_{l_1...l_{h}}$, and we write $w(x_{l_1...l_{h}})=|X_{l_1...l_h}|$.    

%Let $\hat{x}_{k}$ be the subconfiguration obtained from $x$ by removing the bodie $x_k$. The matrix $X_{k}$ of the configuration $\hat{x}_{k}$ is obtained from $X$ by removing the $k$-th column and the last row. The determinant of $X_{k}$ is denoted by $w(\hat{x}_{k})=|X_{k}|$. 
 
 In the next proposition we obtain the equivalence between central configuration equations end Generalized Laura-Andoyer Equations. To our current knowledge, this is the first equivalence result
 for central configurations with dimension $d>2.$
 
 %The following formula will be useful:
%\begin{equation}\label{equseful}
%\tau(i,k)(x_1-x_i)\wedge \cdots \wedge (x_n - x_i ) =w(\hat{x}_k)e_1 \wedge\cdots \wedge e_d,
%\end{equation}
%where $i$ and $k$ are different numbers from $1$ to $d+2$, the factors $(x_k-x_i )$, $(x_i-x_i)$ are omitted, $\{e_1,...,e_d\}$ is the standard basis for $\mathbb{R}^d$,  and 
%$$\tau(i,k)=\begin{cases}
%(-1)^{i+1} &\text{ if } i<k, \\
 %(-1)^i &\text{ if } i>k.
%\end{cases}$$

\begin{proposition}\label{LAG}
Let $x=(x_1,...,x_{d+h})$ be a configuration with $d+h$ bodies in $\mathbb{R}^{d}$. $x$ is a  central   configuration of dimension $d$ if and only if $x$ satisfies the \emph{Laura-Andoyer equations}
\begin{equation}
L_{ijl_{1}\cdots l_{h}}=\sum_{s=1}^{h}(-1)^s m_{l_{s}}(R_{il_{s}}-R_{jl_{s}})\triangle_{l_1 \cdots l_{s-1}l_{s+1}\cdots l_h}=0,
\end{equation}
where 
$$\triangle_{l_1...l_{h}} =(-1)^{\sum_{s=1}^{k}l_{s}}w(x_{l_{1}...l_{h}}).$$
\end{proposition}

\begin{proof}
By the definition of the $\gamma_i$'s
\begin{align}\label{eqN5}
\gamma_i-\gamma_j=&(m_i+m_j)R_{ij}(x_j-x_i)+ \\ \nonumber
                  &\sum_{l\neq i,j}m_l\big(R_{il}(x_l-x_i)-R_{jl}(x_l-x_i)+ R_{jl}(x_{l}-x_i)\big). \nonumber
\end{align}
Consider the $(d-1)$-dimensional exterior product given by
\begin{equation*}
 v^{i}_{l_{1}...l_{h}}=(x_1-x_i)\wedge...\wedge (x_n-x_i) \in \wedge^{d-1}\mathbb{R}^d,
\end{equation*}
in which the terms  $(x_i-x_i),(x_{l_{1}}-x_i),...,(x_{l_{h}}-x_i)$ were omitted. Taking the wedge product of the central configuration equations with $v^{i}_{l_{1}...l_{h}}$, we obtain:

$$(\gamma_i-\gamma_j)\wedge v^{i}_{l_{1}...l_{h}} =\lambda(x_i-x_j)\wedge v^{i}_{l_{1}...l_{h}}=0.$$
By the definition of the $\gamma_i$`s we obtain:
\begin{align}
0=&\sum_{s=1}^{d+h}m_{s}(R_{is}-R_{js})(x_s-x_i)\wedge v^{i}_{l_{1}...l_{h}}\\
 =& \sum_{s=1}^{h}m_{l_{s}}(R_{i l_{s}}-R_{j l_{s}})(x_{l_{s}}-x_i)\wedge v^{i}_{l_{1}...l_{h}}\\
 =& \sum_{s=1}^{h}(-1)^{l_{s}-s-1}m_{l_{s}}(R_{i l_{s}}-R_{j l_{s}})\wedge v^{i}_{l_{1}...l_{s-1}l_{s+1}...l_{h}}
\end{align}
By using the definition of the $\triangle_i$, we get:

$$\sum_{s=1}^{h}(-1)^sm_{l_{s}}(R_{i l_{s}}-R_{j l_{s}})\triangle_{l_{1}...l_{s-1}l_{s+1}...l_{h}}=0$$

Conversely, suppose the Laura-Andoyer equations hold. Then we obtain 
\begin{equation}\label{rij} 
(\gamma_{i}-\gamma_{j})\wedge v^{i}_{l_1...l_h}=0,
\end{equation}
for all distinct $i,j,l_1,...,l_h$ with $i<j$, $l_1<...<l_h$. From these, we will derive that $\gamma_i=\lambda(x_i-c),$ for all $i$.

A key observation is the following identity which can be obtained from the definition of $\gamma_i$'s. For any $q\in \mathbb{R}^{d}$, we have
\begin{equation}\label{anyq} 
\sum_{j=1}^{n}m_{j}(x_j-q)\wedge \gamma_j=0.
\end{equation}

Equation (\ref{rij}) holds for any $i,j \in \{1,...,n\}$. Moreover, fixing $i$, we  have the following equations for all indices $j,p_{1},...,p_{d-2}\in \{1,...,n\}$.
\begin{equation}\label{wedge1} 
m_{p_{1}}(x_{p_{1}}-x_i)\wedge \cdots \wedge m_{p_{d-2}}(x_{p_{d-2}}-x_i)\wedge m_{j}(x_{j}-x_i)\wedge(\gamma_{i}-\gamma_{j})=0. 
\end{equation}

Taking the sum over $j$ for equations (\ref{wedge1}), we find the last two factors become 
$$ \sum_{j=1}^{n}m_{j}(x_j-x_i)\wedge \gamma_i -\sum_{j=1}^{n}m_{j}(x_j-x_i)\wedge \gamma_j,$$ where the first term can be simplified to $M(c-x_{i})\wedge \gamma_i$ and the second term is zero by equation (\eqref{anyq}). So, we have
\begin{equation}\label{wedge2} 
m_{p_{1}}(x_{p_{1}}-x_i)\wedge \cdots \wedge m_{p_{d-2}}(x_{p_{d-2}}-x_i)\wedge(M(c-x_{i})\wedge \gamma_i)=0. 
\end{equation}

Since $x$ have dimension $d$, the set $$\mathcal{B}=\{(x_{p_{1}}-x_i)\wedge \cdots \wedge(x_{p_{d-2}}-x_i): 1\leq p_1 < p_{2}< \cdots p_{d-2}\leq n\}$$ contains a basis of $\wedge^{d-2}\mathbb{R}^{d}$, by equations 
(\ref{wedge2}), $v\wedge M(c-x_{i})\wedge \gamma_i=0$ for all $v\in \mathcal{B}$. Hence $M(c-x_{i})\wedge \gamma_i=0$. Thus, there exist $\lambda_i \in \mathbb{R}$ such that $\gamma_i=\lambda_{i}(x_i-c),$ for all $i$. Next, we will show $\lambda_{i}$'s are all equals.

Since we consider $d\geq2$, $x_{1},...,x_{n}$ are not collinear. So, we can always find distinct $i$ and $j$ such that $x_{i},x_{j}$ and the center of mass $c$ are not collinear. We will show later that this implies $\lambda_{i}=\lambda_{j}$, denoted just by $\lambda$. Placing the remaining $x_{k}$'s into $\mathbb{R}^{d}$, we have either $x_{k},x_{i},c$ are not collinear, or $x_{k},x_{j},c$ are not collinear, or $x_{k}=c$. The first case gives $\lambda_{k}=\lambda_{i}=\lambda$, the second case gives $\lambda_{k}=\lambda_{j}=\lambda$. For the last case, $\gamma_k=\lambda(x_k-c)$ holds.

Next, we show that if $x_{i},x_{j},c$ are not collinear, then $\lambda_{i}=\lambda_{j}$. For convenience, we explain the case Dziobek case of $h=2$. One can obtain the same results for other $h$'s using similar arguments. Here we will use the fact that $x_1,...,x_{(d+2)}$ span $\mathbb{R}^{d}$. We may let $i=1,j=2$. Suppose $\lambda_{1}\neq\lambda_{2}$. Then $(x_{2}-x_{1})\wedge (\gamma_{1}-\gamma_{2})\neq 0$. For $d=2$, it violates $(\gamma_{1}-\gamma_{2})\wedge v^1_{34}=0$ in (\ref{rij}). For $d\geq3$, $x_{2}-x_{1}$ and $\gamma_{1}-\gamma_{2}$ span an affine subspace, denoted by $A_{2}$, of dimension $2$ containing $x_{1},x_{2},c,\gamma_{1},\gamma_{2}$. If $d=3$, (\ref{rij}) gives $(\gamma_{1}-\gamma_{2})\wedge v^1_{34}=(\gamma_{1}-\gamma_{2})\wedge v^1_{35}=(\gamma_{1}-\gamma_{2})\wedge v^1_{45}=0$, which implies $x_{3},x_{4},x_{5}$ are all on the plane $A_{2}$. This contradict to the fact that $x_{1},..,x_{5}$ span $\mathbb{R}^{3}$.

If $d>3$, among $x_{3},...,x_{(d-1)},x_{d},x_{(d+1)},x_{(d+2)}$, at least one of them must be not on $A_{2}$, say $x_{3} \notin A_{2}$. So, $x_{2}-x_{1},x_{3}-x_{1}$ and $\gamma_{1}-\gamma_{2}$ span an affine subspace, denoted by $A_{3}$, of dimension $3$ containing $x_{3}$ and $A_{2}$. Continuing this procedure, we obtain linearly independent vectors $x_{2}-x_{1},x_{3}-x_{1},...,x_{(d-1)}-x_{1}$ and $\gamma_{1}-\gamma_{2}$ spanning an affine subspace, denoted by $A_{(d-1)}$, of dimension $d-1$ containing $x_{d-1}$ and $A_{(d-2)}$. 

On the other hand, from (\ref{rij}), we have $(\gamma_{1}-\gamma_{2})\wedge v^1_{(d+1)(d+2)}=(\gamma_{1}-\gamma_{2})\wedge v^1_{d(d+2)}=(\gamma_{1}-\gamma_{2})\wedge v^1_{d(d+1)}=0$. These imply that $x_{d},x_{(d+1)},x_{(d+2)} \in A_{d-1}$. So, we obtain $x_1,...,x_{(d+2)} \in  A_{d-1}$, contradicting to the fact they span $\mathbb{R}^{d}$. \unskip\nobreak\hfill $\square$

\end{proof}

In the Dziobek case, we have the following:

\begin{proposition}\label{LA}
Let $x=(x_1,...,x_{(d+2)})$ be a configuration. $x$ is a Dziobek central   configuration of dimension $d$ if and only if $x$ satisfies the \emph{Laura-Andoyer equations}
\begin{equation}
L_{ijl_{1}l_{2}}=m_{l_{1}}(R_{il_{1}}-R_{jl_{1}})\triangle_{l_{2}}-m_{l_{2}}(R_{il_{2}}-R_{jl_{2}})\triangle_{l_{1}}=0,
\end{equation}
where $\triangle_{l}=(-1)^{(l+1)}w(x{_l})$, $i,j,l_1$, and $l_2$ are different numbers from $1$ to $d+2$, $i< j$ and $l_1<l_2$. 
\end{proposition}

Laura-Andoyer equations are convenient to study symmetric central configurations. Consider $C_{d,h}$ the set of central configurations with $d+h$ bodies and dimension $d$. Consider $H$ a subgroup of the permutation group $S_{d+h}$. We define an action on the set $\mathcal{X}=\{x_1,...,x_{d+h}\}$ given by $\sigma \cdot x_i= x_{\sigma(i)}$, $\forall \sigma \in H$ . We say that the central configuration  $x=(x_1,...,x_n)$ is strongly symmetric with respect to $H$ if $r_{ij}=r_{\sigma(i)\sigma(j)}$,  $m_i=m_{\sigma(i)}$ and $\triangle_{l_1  \cdots l_{h_1}}=\triangle_{\sigma (l_1) \cdots  \sigma(l_{h-1})}.$ The subgroup $H$ also acts naturally on the set of Laura-Andoyer equations in the following way: $\sigma \cdot L_{ijl_{1}\cdots l_{h}}=L_{\sigma(i)\sigma(j)\sigma(l_{1})\cdots \sigma(l_{h})}$. 

Let $C^{H}_{d,h}$ be the set central configurations with $d+h$ bodies and dimension $d$ strongly symmetric to $H$. Then, 

$$\sigma \cdot L_{ijl_{1}\cdots l_{h}}=L_{ijl_{1}\cdots l_{h}},$$ 

 $\forall x\in C^{H}_{d,h}$ and $\sigma \in H$. Note that the representative of the orbits of the action of $H$ in the Laura-Andoyer equations are a symmetry-reduced set of algebraic equations for $C^{H}_{d,h}$.

We will apply this approach in detail to obtain a symmetry-reduced set of Laura-Andoyer equations for the $(2,d)-$cc's.

%\begin{proof}
%By the definition of the $\gamma_i$'s
%\begin{align*}\label{eqN5}
%\gamma_i-\gamma_j=&(m_i+m_j)R_{ij}(x_j-x_i)+ \\ \nonumber
%                  &\sum_{k\neq i,j}m_k\big(R_{ik}(x_k-x_i)-R_{jk}(x_k-x_i)+ R_{jk}(x_{j}-x_i)\big). \nonumber
%\end{align*}

%Consider the $(d-1)-$dimensional multivectors of the form
%$$v_{ikl} = (x_1-x_i)\wedge \cdots \wedge (x_n-x_i),$$
%where $i, k,l\in\{1,...,n\}$, $k<l$ and the  factors $(x_i-x_i)$,$(x_k-x_i)$ and $(x_l-x_i)$ are omitted. Note that 
%$$(\gamma_i-\gamma_j)\wedge v_{ikl}=\lambda(x_i-x_j)\wedge v_{ikl}=0.$$
 
%To derive the Laura-Andoyer equations we need to compute the left side of the equation above.  We have three cases to consider:  $k<l<i$, $k<i<l%$ and $i<k<l$. We will do only the first case. The other cases are analogous.
%\begin{align*}
%0=(\gamma_{i}-\gamma_{j})\wedge v_{ikl}=&m_{k}(R_{ik}-R_{jk})(x_{k}-x_{i})\wedge v_{ikl}+\\
%                                      & m_l(R_{il}-R_{jl})(x_l-x_i)\wedge v_{ikl}=0.\end{align*}

%Applying formula \eqref{equseful} we get
%\begin{align*}
%0=(\gamma_{i}-\gamma_{j})\wedge v_{ikl}=&m_k(R_{ik}-R_{jk})(-1)^{k-1}(-1)^{i}w(\hat{x}_{l})e_1\wedge...\wedge e_{d}+\\ \nonumber
%&m_l(R_{il}-R_{jl})(-1)^{l-2}(-1)^{i}w(\hat{x}_{k})e_1\wedge...\wedge e_{d} \nonumber.
%\end{align*}
%This implies that
%$$m_k(R_{ik}-R_{jk})(-1)^{(k-1+i+l+1)}\triangle_{l}+m_l(R_{il}-R_{jl})(-1)^{(l-2+i+k-1)}\triangle_{k}=0.$$
%The result is proved. 
%\end{proof}

The set of the Laura-Andoyer equations satisfied by a Dziobek configuration is given by
$$LA_{d+2}=\{L_{ijl_1l_2}: i,j,l_1,l_2 \text{ are distinct}, i<j, l_1<l_2\}.$$
Let $ \#S$ denote the cardinality of the set $S$. Note that $\#LA_{d+2}=\binom{d+2}{2}\binom{d}{2}.$ In this way, we get a huge number of equations when  $d$ is big.  In the case of $(2,d)$-cc's we reduce the number of different equations to only two for every choice of $d\geq2$.  Let us denote by $S_d$ the group of the permutations on the set $\{3,...,d+2\}$.  If $x$ is a $(2,d)$-cc then $S_d$ is a subgroup of the symmetric group on the set $\mathcal{X}=\{x_1,...,x_{d+2}\}$.

By definition, in a $(2,d)$-cc we have $m_i=m_{\sigma{(i)}}$ and $R_{ij}=R_{\sigma(i)\sigma(j)}$, for all $\sigma \in S_{d}$ . To conclude that a $(2,d)$-cc is strongly symmetric with respect to $S_d$ we will show that $\triangle_{i}=\triangle_{\sigma(i)}$  $\forall$ $\sigma \in S_{d}$. We need the following Dziobek equations.

\begin{proposition}
If $x$ is a Dziobek configuration then there exists  a constant $\eta$ and  a nonzero constant $\tau$, such that

\begin{equation}\label{Eqdz}
R_{ij}-\eta=\tau\frac{\triangle_i}{m_i}\frac{\triangle_j}{m_j},    
\end{equation}

for which $\triangle_i$ is as in proposition $\ref{LA}$.
\end{proposition}

\begin{proof}
In this proof, we follow \cite{M}, but our notation is different. By elementary computations, one can check that $\eqref{eqc}$ in equivalent to the following system of equations:

\begin{equation}
\sum_{\begin{subarray}{c}j=1\\ j\neq i\end{subarray}}^{n}\left(R_{ij}-\eta\right)(x_j-x_i)=0, \qquad 1\leq i \leq n,
\end{equation}

for which $\eta=\frac{\lambda}{M}$. Consider the $(d-1)-$ dimensional exterior product 

\begin{equation*}
v^{i}_{l_{1}l_{2}}=(x_1-x_i)\wedge...\wedge (x_{d+2}-x_i) \in \wedge^{d-1}\mathbb{R}^d,
\end{equation*}
in which $i$, $l_1$, and $l_2$ are three distinct indices from $1$ to $d+2$ and the terms  $(x_i-x_i),(x_{l_{1}}-x_i),(x_{l_{2}}-x_i)$ were omitted. By using properties of the exterior product
like the proposition \ref{LA}, we get the equations

\begin{equation}\label{msd}
m_{l_{1}}\left(R_{il_{1}}-\eta\right)\triangle_{l_{2}}=m_{l_2}\left(R_{il_{2}}-\eta\right)\triangle_{l_1}.
\end{equation}

for which $\triangle_{l}=(-1)^{(l+1)}w(x{_l})$. Note that, equations \eqref{msd} implies that the matrices 

$$M_i=\left(\begin{array}{cccccc}
    R_{i1}-\eta &                  \ldots & R_{i(i-1)}-\eta                   & R_{i(i+1)}-\eta                    & \ldots & R_{in}  -\eta \\
    \frac{\triangle_{1}}{m_1}   & \ldots & \frac{\triangle_{i-1}}{m_{i-1}}& \frac{\triangle_{i+1}}{m_{i+1}}  & \ldots & \frac{\triangle_{n}}{m_{n}}
\end{array}\right)$$

have rank $1$ for $i=1,\cdots,n.$ Hence, there exists constant $c_i$ for which 

$$ R_{ij}-\eta=c_i\frac{\triangle_{j}}{m_j}$$

with $1\leq i < j < n$. Since $R_{ij}=R_{ji},$ it follows that

$$c_i\frac{\triangle_{j}}{m_j}=c_j\frac{\triangle_{i}}{m_i}.$$

Then, for $i=1,\cdots, n$ the matrix

$$C_i=\left(\begin{array}{ccc}
    c_1 &                  \ldots &  c_n \\
    \frac{\triangle_{1}}{m_1}   & \ldots & \frac{\triangle_{n}}{m_{n}}
\end{array}\right)$$
also have rank $1$. Hence there is a constant $\tau$ such that $c_i=\tau\frac{\triangle_{i}}{m_i}$. This follows equations $\eqref{Eqdz}$.
We claim that $\tau\neq 0$. Otherwise, all $R_{ij}$ are equal and the configuration $x$ is a regular $(d-1)-$dimensional simplex. This contradicts $x$ being a Dziobek configuration.  \unskip\nobreak\hfill $\square$

\end{proof}

\begin{lemma}\label{Delta}
If $x$ is a configuration such that there are $d$ bodies with unit masses at the vertices of a regular $(d-1)$-dimensional simplex of unit edge length, and two more bodies with masses $s,k$ on the line passing through the center of the simplex and being orthogonal to it. Then $\triangle_{i}=\triangle_{\sigma(i)},$ $\forall~\sigma \in S_{d}.$
\end{lemma}
\begin{proof}It is easy to check that $\triangle_i\neq 0$ for $i=1,...,d+2$.
From the Dziobek equations we get $$\triangle_{i}\triangle_{j}=\triangle_{l_1}\triangle_{l_2},$$
for all $i,j,l_1,l_2\in\{3,...,d+2\}$ with $i<j$ and $l_1<l_2$.

Hence,
 $$\triangle_i(\triangle_{l_1}-\triangle_{l_2})=0,$$ 
 for all $i,l_1,l_2\in\{3,...,d+2\}$ such that $i\neq l_1$, $i\neq l_2$ and $l_1<l_2$. Since $\triangle_{i}\neq 0$ for all $i=3,...,d+2$ the result is proved. \unskip\nobreak\hfill $\square$
 \end{proof}

Note that, by definition a $(2,d)$-cc satisfies the geometric conditions of lemma $\ref{Delta}.$ Then, we have the following:

\begin{lemma}
$L_{ijl_1l_2}=L_{\sigma{(i)}\sigma{(j)}\sigma{(l_1)}\sigma{(l_2)}}$ for all $L_{ijl_1l_2}\in LA_{d+2}.$
\end{lemma}
%\begin{proof}
%By the lemma $\ref{Delta}$  we have
%\begin{align*}
%&m_{\sigma(l_1)}(R_{\sigma(i)\sigma(l_1)}- R_{\sigma(j)\sigma(l_1)})\triangle_{\sigma(l_2)}-m_{\sigma(l_2)}(R_{\sigma(i)\sigma(l_2)}-R_{\sigma(j)\sigma(l_2)})\triangle_{\sigma(l_1)}= \\
%&m_{l_1}(R_{il_1}-R_{jl_1})\triangle_{l_2}-m_{l_2}(R_{il_2}-R_{jl_2})\triangle_{l_1}.
%\end{align*}
%This establishes the formula.  \unskip\nobreak\hfill $\square$
%\end{proof}

  We denote the orbits of of the action of $S_d$ on $LA_{d+2}$ by $$\mathcal{O}(L_{ijl_1l_2})=\{L_{\sigma(i)\sigma(j)\sigma(l_1)\sigma(l_2)}:\sigma \in S_d\},$$ and the set of the respective orbits by $LA_{d+2}/S_d$.

\begin{proposition} If $d\geq 2$, the only two nontrivial equations of $LA_{d+2}$ are $L_{1324}$ and $L_{2314}.$
\end{proposition}

\begin{proof} 

If $d=2$, $\# LA_4=2$ and if $d=3$, $\# LA_5=30$. It is easy to check directly that in this cases the only two nontrivial equations are $L_{1324}$ and $L_{2314}.$ If $d\geq 4$ we claim that 

\begin{align*} 
LA_{d+2}/S_d=&\{\mathcal{O}(L_{3456}),\mathcal{O}(L_{3412}),\mathcal{O}(L_{3415}),\mathcal{O}(L_{3425}),\mathcal{O}(L_{1345}),\mathcal{O}(L_{2345}),\\
&\mathcal{O}(L_{1234}),\mathcal{O}(L_{1324}),\mathcal{O}(L_{2314})\}.
\end{align*}
In fact, we have
\begin{align*}
 \#\mathcal{O}(L_{3456})=\binom{d}{2}\binom{d}{2}, \qquad & \#\mathcal{O}(L_{3412})=\binom{d}{2}, &\#\mathcal{O}(L_{3415})=(d-2)\binom{d}{2},\\
 \#\mathcal{O}(L_{3425})=(d-2)\binom{d}{2},\qquad &\#\mathcal{O}(L_{1345})=d\binom{d-1}{2},&\#\mathcal{O}(L_{2345})=d\binom{d-1}{2},\\
 \#\mathcal{O}(L_{1234})=\binom{d}{2},\qquad& \#\mathcal{O}(L_{1324})=d(d-1),&\#\mathcal{O}(L_{2314})=d(d-1).
\end{align*}
The sum of the cardinalities of the orbits described above  is equal to $\binom{d+2}{2}\binom{d}{2}=\#LA_{d+2}$, which proves our claim. Checking case by case directly, we get that the only nontrivial equations are
$$L_{1324}=m_2(R_{12}-R_{23})\triangle_{4}-m_4(R_{14}-R_{34})\triangle_{2}=0,$$
$$L_{2314}=m_1(R_{12}-R_{13})\triangle_{4}-m_4(R_{24}-R_{34})\triangle_{1}=0.$$  \unskip\nobreak\hfill $\square$
\end{proof}

We notice that equations $L_{1324}$ and $L_{2314}$ are typical in the theory of Dziobek central configurations. For example, they appear in the works of  Dzoiobek \cite{OD} and Schmidt \cite{S}, among others. Our procedure differs from our predecessors since it uses the group symmetry of the $(2,d)-$cc's to obtain only two equations that describe the class of central configurations studied. Combined with the proposition \ref{LAG}, this approach can be used to obtain adequate Laura-Andoyer systems for classes of symmetrical central configurations with $d+h$ bodies and dimension $d$ with $h>2.$ 

Now, we are ready to derive our polynomial system form $L_{1324},L_{2314}$.

\begin{proposition}
Under the assumptions above, $L_{1324}=L_{2314}=0$ becomes 
\begin{equation}\label{algsys}
  \begin{cases}
   k((z-w)^{2a}-(\frac{d-1}{2d}+w^{2})^{a})(z-w)+((\frac{d-1}{2d}+z^{2})^{a}-1)zd=0,\\
   s((z-w)^{2a}-(\frac{d-1}{2d}+z^{2})^{a})(z-w)-((\frac{d-1}{2d}+w^{2})^{a}-1)wd=0.\\
   \end{cases}
\end{equation}

\end{proposition}
Note this system generalizes the systems in \cite{T} for $d=2,3$.
\begin{proof} 
Recall that, for a Dziobek configuration $x=(x_1,...,x_{d+2})$ where $x_{i}\in \mathbb{R}^{d}$. The matrix of the configuration is

$$X=\left(\begin{array}{ccc}
    1 & \ldots & 1 \\
    x_{11} & \ldots & x_{1(d+2)} \\
    \vdots &\ldots& \vdots\\
    x_{d1} & \ldots & x_{d(d+2)}\\
    0 & \cdots &0
    \end{array}\right).$$

And, $\triangle_{k}=(-1)^{(k+1)}|X_{k}|$, where $X_{k}$ is the submatrix obtained from $X$ by removing the $k$-th column and the last row. The determinant $|X_{k}|$ is $d!$ times the signed volume of the $d$-dimensional convex hull determined by all points except for the $k$-body. 

Therefore, we compute $\triangle_{4},\triangle_{2},\triangle_{1}$, respectively as follow.  
\[
\begin{aligned}
&(-1)^{4+1}\begin{vmatrix}

    1 &1&1&1& \dots & 1 \\
    z &w&0&0& \dots & 0 \\
    0 &0&\vdots&\vdots& \dots & \vdots\\
   \vdots &\vdots&\delta_{1}&\delta_{3}& \dots & \delta_{d}\\
    0 &0&\vdots&\vdots& \dots & \vdots\\
    \end{vmatrix}=(-1)\begin{vmatrix}

    1 &0&1&1& \dots & 1 \\
    z &w-z&0&0& \dots & 0 \\
    0 &0&\vdots&\vdots& \dots & \vdots\\
   \vdots &\vdots&\delta_{1}&\delta_{3}& \dots & \delta_{d}\\
    0 &0&\vdots&\vdots& \dots & \vdots\\
    \end{vmatrix}=(z-w)\begin{vmatrix}
    1 &1&1& \dots & 1 \\
    0 &\vdots&\vdots& \dots & \vdots\\
   \vdots &\delta_{1}&\delta_{3}& \dots & \delta_{d}\\
    0 &\vdots&\vdots& \dots & \vdots\\
    \end{vmatrix} \\
    &=(z-w)\begin{vmatrix}
   \delta_{1}&\delta_{3}& \dots & \delta_{d}\\
    \end{vmatrix}.
\end{aligned}
\]

\[
\begin{aligned}
&(-1)^{2+1}\begin{vmatrix}
    1 &1&1& \dots & 1 \\
    z &0&0& \dots & 0 \\
    0 &\vdots&\vdots& \dots & \vdots\\
   \vdots &\delta_{1}&\delta_{2}& \dots & \delta_{d}\\
    0 &\vdots&\vdots& \dots & \vdots\\
    \end{vmatrix}=z\begin{vmatrix}
    1&1& \dots & 1 \\
   \delta_{1}&\delta_{2}& \dots & \delta_{d}\\
    \end{vmatrix}=z\begin{vmatrix}
    0&1&0& \dots & 0 \\
   \delta_{1}-\delta_{2}&\delta_{2}& \delta_{3}-\delta_{2}& \dots & \delta_{d}-\delta_{2}\\
    \end{vmatrix} \\
    &=(-z)\begin{vmatrix}
   \delta_{1}-\delta_{2}&\delta_{3}-\delta_{2}& \dots & \delta_{d}-\delta_{2}\\
    \end{vmatrix}.
\end{aligned}
\]

\[
\begin{aligned}
&(-1)^{1+1}\begin{vmatrix}
    1 &1&1& \dots & 1 \\
    w &0&0& \dots & 0 \\
    0 &\vdots&\vdots& \dots & \vdots\\
   \vdots &\delta_{1}&\delta_{2}& \dots & \delta_{d}\\
    0 &\vdots&\vdots& \dots & \vdots\\
    \end{vmatrix}=(-w)\begin{vmatrix}
    1&1& \dots & 1 \\
   \delta_{1}&\delta_{2}& \dots & \delta_{d}\\
    \end{vmatrix}=z\begin{vmatrix}
    0&1&0& \dots & 0 \\
   \delta_{1}-\delta_{2}&\delta_{2}& \delta_{3}-\delta_{2}& \dots & \delta_{d}-\delta_{2}\\
    \end{vmatrix} \\
    &=w\begin{vmatrix}
   \delta_{1}-\delta_{2}&\delta_{3}-\delta_{2}& \dots & \delta_{d}-\delta_{2}\\
    \end{vmatrix}.
\end{aligned}
\]

Now, $\begin{vmatrix}
   \delta_{1}-\delta_{2}&\delta_{3}-\delta_{2}& \dots & \delta_{d}-\delta_{2}\\
    \end{vmatrix}$ is $(d-1)!$ times the signed volume of the $(d-1)$-dimensional regular simplex with unit side length. Since the center of the simplex is at the origin, $\begin{vmatrix}
   \delta_{1}&\delta_{3}& \dots & \delta_{d}\\
    \end{vmatrix}$ is just $\frac{1}{d}$ of $\begin{vmatrix}
   \delta_{1}-\delta_{2}&\delta_{3}-\delta_{2}& \dots & \delta_{d}-\delta_{2}\\
    \end{vmatrix}$. In conclusion, we have $\triangle_{4}=(z-w)\begin{vmatrix}
   \delta_{1}&\delta_{3}& \dots & \delta_{d}\\
    \end{vmatrix},\triangle_{2}=(-zd)\begin{vmatrix}
   \delta_{1}&\delta_{3}& \dots & \delta_{d}\\
    \end{vmatrix}$, and $\triangle_{1}=wd\begin{vmatrix}
   \delta_{1}&\delta_{3}& \dots & \delta_{d}\\
    \end{vmatrix}. $ 

Next, using the fact that  the circumradius of the regular $(d-1)$-simplex with unit edge length in $\mathbb{R}^{d-1}$ is $\sqrt{\frac{d-1}{2d}}$, we have $R_{13}=R_{14}=(\frac{d-1}{2d}+z^{2})^{a}$, $R_{23}=R_{24}=(\frac{d-1}{2d}+w^{2})^{a}$. Since $z>w$, we have $R_{12}=(z-w)^{2a}$. Since the simplex has unit edge length, $R_{34}=1$. With all these information together with $m_{1}=s, m_{2}=k, m_{4}=1$, we obtain system (\ref{algsys}) from $L_{1324}=L_{2314}=0$.  \unskip\nobreak\hfill $\square$
\end{proof}

As a final remark, we observe that our proof is much longer than the one given by Leandro in \cite{L}. Despite that,  we include it in our work because we believe that it works to simplify central configuration equations in cases with symmetric central configurations with dimension $d\geq 2$ and $d+h$ bodies with $h \geq 3$, and symmetry group more complicated than $S_n$. In future work, the authors pretend to explore some of these cases.

\begin{acknowledgements}
\thanks{ We would like to express our sincere gratitude to the anonymous referees for their careful reading of our manuscript and their many insightful comments, questions and suggestions. This research was partly supported by the Ministry of Science and Technology of Taiwan under the grants  MOST 108-2115-M-005-004-MY2 (Ya-Lun Tsai) and MOST 107-2811-M-007-004 (Thiago Dias).}
\end{acknowledgements}


\begin{thebibliography}{99}

% etc

\bibitem{A}
A. Albouy, 
\emph{On a paper of Moeckel on central configurations}, 
Regular and Chaotic Dynamics \textbf{8}, no. 2, 133-142 (2003)

\bibitem{A2}
A. Albouy,  Y. Fu, and S. Sun, 
\emph{Symmetry of planar four-body convex central configurations}, 
Proceedings of the Royal Society A: Mathematical, Physical and Engineering Sciences \textbf{464}, no 2093,  1355-1365 (2008)

\bibitem{AC}
A. Chenciner, 
\emph{Are nonsymmetric balanced configurations of four equal masses virtual or real?}, 
Regular and Chaotic Dynamics \textbf{22}, issue 6, 677-687 (2017)

\bibitem{AC1}
A. Chenciner, 
\emph{Symmetric $4$-body balanced configurations: the equal mass case}, 
prinprint (2016)

\bibitem{ACS}
A. Albouy, H. E. Cabral, A. A. Santos,
\emph{Some problems on the classical $n$-body problem}, 
Celestial Mech. Dynam. Astronom.  \textbf{113},  no. 4, 369–375 (2012)

\bibitem{ASV}
M. Alvarez-Ramírez, A. A. Santos, C. Vidal,
\emph{On Co-Circular Central Configurations in the Four and Five Body-Problems for Homogeneous Force Law}
J. Dyn. Diff. Equat. \textbf{25} 269–290 (2013)

\bibitem{BPR}
S. Basu, R. Pollack, M.-F. Roy, 
\emph{Algorithms in Real Algebraic Geometry, Algorithms},
Comput. Math. 10, Springer-Verlag, Berlin (2003)

\bibitem{C}
G.E. Collins, 
\emph{Quantifier elimination for real closed fields by cylindrical algebraic decomposition}.  Automata theory and formal languages (Second GI Conf., Kaiserslautern, 1975),  134–183. Lecture Notes in Comput. Sci., Vol. 33, Springer, Berlin (1975)

\bibitem{CK}
G.E. Collins, W. Krandick,
\emph{An efficient algorithm for infallible polynomial complex root isolation}
In: Wang, Paul S.(Ed.): Proceedings of ISSAC'92, 189-194 (1992)

\bibitem{CLO}
D. Cox, J. Little, D. O'Shea,
\emph{Ideals, Varieties and Algorithms, an Introduction to Computational Algebraic Geometry and Commutative Algebra},
Undergrad. Texts Math., Springer, New York (1992)

\bibitem{D}
T. Dias,
\emph{New equations for central configurations and generic finiteness}, 
Proc. Amer. Math. Soc. \textbf{145}, no. 7, 3069–3084 (2017) 


\bibitem{OD}
 O. Dziobek, 
 \emph{Ueber einen merkwürdigen Fall des Vielkörperproblems.} 
 Astronomische Nachrichten \textbf{152}, issue 3, 33-46. (1900)

\bibitem{FM}
A. C. Fernandes,  L. F. Mello. 
\emph{Andoyer equations for noncollinear planar central configurations},
Turkish Journal of Mathematics, \textbf{41}, no 3, 515-523 (2017)

\bibitem{H}
Y. Hagihara,
\emph{Celestial Mechanics},
V 1 chap. 3. The MIT Press, Cambridge, MA (1970)

\bibitem{HM1}
M. Hampton, R. Moeckel,
\emph{Finiteness of relative equilibria of the four-body problem},
Inv. Math., \textbf{163}, 289-312 (2006).

\bibitem{HM2}
M. Hampton, R. Moeckel,
\emph{Finiteness of stationary configurations of the four-vortex problem},
Trans. Amer. Math. Soc., \textbf{361}, no. 3, 1317–1332 (2009).

\bibitem{HS}
M. Hampton, and M. Santoprete. 
\emph{Seven-body central configurations: a family of central configurations in the spatial seven-body problem}, 
Celestial Mechanics and Dynamical Astronomy, \textbf{99}, 293-305 (2007) 


\bibitem{L}
E. S. G. Leandro,
\emph{Finiteness and bifurcations of some symmetrical classes of central configurations},
Arch. Ration. Mech. Anal., \textbf{167},  no. 2, 147-177 (2003).

\bibitem{L1}
E. S. G. Leandro,
\emph{On the Dziobek configurations of the restricted $(N+1)$-body problem with equal masses},
Discrete and Continuous Dynamical Systems, \textbf{1},  no. 4, 589–595 (2008). 

\bibitem{L2}
J. Llibre,
\emph{A note on the Dziobek central configurations
},
Proc. Amer. Math. Soc., \textbf{143}, no. 8, 3587–3591  (2015).

\bibitem{ME}
G. Meyer,  
\emph{Solutions voisines des solutions de Lagrange dans le probleme des n corps}, 
 Annales de l'Observatoire de Bordeaux \textbf{17}, 77-252 (1933)

\bibitem{M}
R. Moeckel,
\emph{Generic finiteness for Dziobek configurations}, 
Trans. Amer. Math. Soc. \textbf{353}, no. 11, 4673–4686 (2001). 

\bibitem{M2}
R. Moeckel,
\emph{Central configurations}.  
Central configurations, periodic orbits, and Hamiltonian systems,  105–167, 
Adv. Courses Math. CRM Barcelona, Birkhäuser/Springer, Basel (2015).

\bibitem{O}
K. A. O’Neil, 
\emph{Stationary configurations of point vortices}, 
Trans. Amer. Math. Soc. \textbf{302}, no. 2, 383–425 (1987).


\bibitem{B}
B. Pan,
\emph{Some case studies of central configurations and central measures},
Phd Thesis, Department of Mathematics, National Tsing Hua University (2017) \url{https://hdl.handle.net/11296/c6vg2x}
\bibitem{S}


D. Saari,
\emph{Central configurations, a problem for the 21st century}
https://www.math.uci.edu/~dsaari/BAMA-pap.pdf


\bibitem{S}
D. Schmidt, 
\emph{Central configurations in $\mathbb{R}^2$ and $\mathbb{R}^3$.} 
Contemp. Math \textbf{81}, 59-76 (1988).

\bibitem{SS}
S. Smale, 
\emph{Mathematical problems for the next century} 
Mathematical Intelligencer \textbf{20}, 7–15 (1998).

\bibitem{SMPV}
A. A. Santos, M. Marchesin, E. Pérez-Chavela, C. Vidal,
\emph{Continuation and bifurcations of concave central configurations in the four and five body-problems for homogeneous force laws}
J. Math. Anal. Appl.  \textbf{446}, 1743–1768 (2017).

\bibitem{T}
Y. Tsai,
\emph{Some enumeration problems on central configurations at the bifurcation points},
Acta Appl. Math., \textbf{155}, 99-112 (2018).

\bibitem{T2}
Y. Tsai,
\emph{Dziobek configurations of the restricted (N + 1)-body problem with equal masses},
J. Math. Phys., \textbf{53}, 072902 (2012). 

\bibitem{link}
Mathematica notebook,
\emph{Generalized Laura-Andoyer equations and the enumeration of some symmetrical classes of Dziobek configurations.nb}, https://github.com/thiagodiasoliveira/BifurcationSurface


\bibitem{link1}
Mathematica tutorial,
\emph{Gr\"obner basis computation}, https://reference.wolfram.com/language/tutorial/ComplexPolynomialSystems.html

\bibitem{link2}
Mathematica tutorial,
\emph{cylindrical algebraic decomposition}, https://reference.wolfram.com/language/tutorial/RealPolynomialSystems.html

\end{thebibliography}


\end{document}